\documentclass[12pt,twoside]{amsart}
\usepackage[latin1]{inputenc}
\usepackage{amsmath, amsthm, amscd, amsfonts, amssymb, graphicx}
\usepackage[bookmarksnumbered, plainpages]{hyperref}

\newtheorem{thm}{Theorem}[section]

\newtheorem{lem}[thm]{Lemma}
\newtheorem{prop}[thm]{Proposition}

\numberwithin{equation}{section}

\begin{document}

\title{{Several new Witten rigidity theorems for elliptic genus}}

\author{ Jianyun Guan; Kefeng Liu and Yong Wang*\\
 }

\date{}

\thanks{{\scriptsize
\hskip -0.4 true cm \textit{2010 Mathematics Subject Classification:}
58C20; 57R20; 53C80.
\newline \textit{* Corresponding author.}}}

\maketitle

\begin{abstract}
Using the Liu's method, we prove a new Witten rigidity theorem of elliptic genus of twisted Dirac operators in even dimensional spin manifolds under the circle action. Combined with the Han-Yu's method, we prove the Witten rigidity theorems of elliptic genus of twisted Toplitz operators of odd-dimensional spin manifolds under the circle action. Moreover, we have obtained several similar Witten rigidity theorems of elliptic genus.
\end{abstract}

\vskip 0.2 true cm

\section{Introduction}

The study of the rigidity of the signature operator, the Dolbeault operator and the Dirac operator on different manifolds is an interesting topic in geometry and mathematical physics. Around 1982, motivated by physics, Witten proved the rigidity of the twisted Dirac operator $D\otimes TX$ for compact homogeneous spin manifolds. In 1988, Witten derived a series of twisted Dirac operators on the free loop space $LM$ of a spin manifold $M$\cite{WE1}. In this work, Witten was surprised to discover that the elliptic genus, constructed topologically by Landweber-Stone\cite{LS} and Ochanine\cite{OS}, is the index of one of the operators. Inspired by physics, Witten conjectured that these elliptic operators should be rigid.
Further, in 1989, the Witten conjecture was first proved by Taubes\cite{TH} and Bott-Taubes\cite{BT}. Hirzebruch \cite{HF} and Krichever\cite{KM} proved Witten conjecture for almost complex manifold case. In 1996, Liu used modular invariance to give a simple and uniform proof, as well as various extensive generalizations of Witten conjecture\cite{Li2,Li3}. And Liu also established several new vanishing theorems in this work. To further generalize, in 2000, Dessai established the rigidity and vanishing theorems for spinc case. Liu-Ma\cite{Li4,Li5} and Liu-Ma-Zhang\cite{Li6,Li7} generalized the rigidity and vanishing theorems to the family case on the levels of equivariant Chern character and of equivariant K-theory.

On the other hand, in 2009, Liu and Wang established the rigidity for twisted Toeplitz operators associated to the Witten bundles by approach\cite{Li2,Li3} under the assumption that the fixed point sets of the group action are 1-dimensional. This work is the first to study the Witten-type rigidity of Toplitz operators\cite{Li8}. In 2015, Han and Yu extend the rigidity and vanishing properties for twisted Toeplitz operators to the cases of fixed points of general dimensions\cite{HY}. In \cite{GW,GW1}, we construct some new $SL(2,\mathbb{Z})$ modular forms, $\Gamma^0(2)$ modular forms and $\Gamma_0(2)$ modular forms in spin manifolds by the modular invariance of characteristic forms. Moreover, we get some divisibility results of index of the twisted Dirac operators and the twisted Toeplitz operators on spin manifolds. The purpose of this paper is to study the rigidity properties of these elliptic genus associated with twisted Dirac operators and twisted Toplitz operators in \cite{GW,GW1} in the case of fixed points of general dimension by using Liu's method and Han-Yu's method.

The structure of this paper is briefly described below: In Section 2, we have introduce some definitions and basic concepts that we will use in the paper. In Section 3, we prove the Witten-type rigidity of twisted Dirac operators in even dimensions. Finally, in section 4, we prove the Witten-type rigidity of twisted Toplitz operators in odd dimensions.\\

\section{Characteristic Forms and Modular Forms}
\quad The purpose of this section is to review the necessary knowledge on
characteristic forms and modular forms that we are going to use.\\

 \noindent {\bf  2.1 characteristic forms }\\
 \indent Let $M$ be a Riemannian manifold.
 Let $\nabla^{ TM}$ be the associated Levi-Civita connection on $TM$
 and $R^{TM}=(\nabla^{TM})^2$ be the curvature of $\nabla^{ TM}$.
 Let $\widehat{A}(TM,\nabla^{ TM})$ and $\widehat{L}(TM,\nabla^{ TM})$
 be the Hirzebruch characteristic forms defined respectively by (cf. \cite{Zh})
\begin{equation}
   \widehat{A}(TM,\nabla^{ TM})={\rm
det}^{\frac{1}{2}}\left(\frac{\frac{\sqrt{-1}}{4\pi}R^{TM}}{{\rm
sinh}(\frac{\sqrt{-1}}{4\pi}R^{TM})}\right),
\end{equation}
 \begin{equation}
     \widehat{L}(TM,\nabla^{ TM})={\rm
 det}^{\frac{1}{2}}\left(\frac{\frac{\sqrt{-1}}{2\pi}R^{TM}}{{\rm
 tanh}(\frac{\sqrt{-1}}{4\pi}R^{TM})}\right).
 \end{equation}
   Let $E$, $F$ be two Hermitian vector bundles over $M$ carrying
   Hermitian connection $\nabla^E,\nabla^F$ respectively. Let
   $R^E=(\nabla^E)^2$ (resp. $R^F=(\nabla^F)^2$) be the curvature of
   $\nabla^E$ (resp. $\nabla^F$). If we set the formal difference
   $G=E-F$, then $G$ carries an induced Hermitian connection
   $\nabla^G$ in an obvious sense. We define the associated Chern
   character form as
   \begin{equation}
       {\rm ch}(G,\nabla^G)={\rm tr}\left[{\rm
   exp}(\frac{\sqrt{-1}}{2\pi}R^E)\right]-{\rm tr}\left[{\rm
   exp}(\frac{\sqrt{-1}}{2\pi}R^F)\right].
   \end{equation}
   For any complex number $t$, let
   $$\wedge_t(E)={\bf C}|_M+tE+t^2\wedge^2(E)+\cdots,~S_t(E)={\bf
   C}|_M+tE+t^2S^2(E)+\cdots$$
   denote respectively the total exterior and symmetric powers of
   $E$, which live in $K(M)[[t]].$ The following relations between
   these operations hold,
   \begin{equation}
       S_t(E)=\frac{1}{\wedge_{-t}(E)},~\wedge_t(E-F)=\frac{\wedge_t(E)}{\wedge_t(F)}.
   \end{equation}
   Moreover, if $\{\omega_i\},\{\omega_j'\}$ are formal Chern roots
   for Hermitian vector bundles $E,F$ respectively, then
   \begin{equation}
       {\rm ch}(\wedge_t(E))=\prod_i(1+e^{\omega_i}t)
   \end{equation}
   Then we have the following formulas for Chern character forms,
   \begin{equation}
       {\rm ch}(S_t(E))=\frac{1}{\prod_i(1-e^{\omega_i}t)},~
{\rm ch}(\wedge_t(E-F))=\frac{\prod_i(1+e^{\omega_i}t)}{\prod_j(1+e^{\omega_j'}t)}.
   \end{equation}
\indent If $W$ is a real Euclidean vector bundle over $M$ carrying a
Euclidean connection $\nabla^W$, then its complexification $W_{\bf
C}=W\otimes {\bf C}$ is a complex vector bundle over $M$ carrying a
canonical induced Hermitian metric from that of $W$, as well as a
Hermitian connection $\nabla^{W_{\bf C}}$ induced from $\nabla^W$.
If $E$ is a vector bundle (complex or real) over $M$, set
$\widetilde{E}=E-{\rm dim}E$ in $K(M)$ or $KO(M)$.\\

\noindent{\bf 2.2 Some properties about the Jacobi theta functions
and modular forms}\\
   \indent We first recall the four Jacobi theta functions are
   defined as follows( cf. \cite{Ch}):
   \begin{equation}
      \theta(v,\tau)=2q^{\frac{1}{8}}{\rm sin}(\pi
   v)\prod_{j=1}^{\infty}[(1-q^j)(1-e^{2\pi\sqrt{-1}v}q^j)(1-e^{-2\pi\sqrt{-1}v}q^j)],
   \end{equation}
\begin{equation}
    \theta_1(v,\tau)=2q^{\frac{1}{8}}{\rm cos}(\pi
   v)\prod_{j=1}^{\infty}[(1-q^j)(1+e^{2\pi\sqrt{-1}v}q^j)(1+e^{-2\pi\sqrt{-1}v}q^j)],
\end{equation}
\begin{equation}
    \theta_2(v,\tau)=\prod_{j=1}^{\infty}[(1-q^j)(1-e^{2\pi\sqrt{-1}v}q^{j-\frac{1}{2}})
(1-e^{-2\pi\sqrt{-1}v}q^{j-\frac{1}{2}})],
\end{equation}
\begin{equation}
   \theta_3(v,\tau)=\prod_{j=1}^{\infty}[(1-q^j)(1+e^{2\pi\sqrt{-1}v}q^{j-\frac{1}{2}})
(1+e^{-2\pi\sqrt{-1}v}q^{j-\frac{1}{2}})],
\end{equation}
 \noindent
where $q=e^{2\pi\sqrt{-1}\tau}$ with $\tau\in\textbf{H}$, the upper
half complex plane. Let
\begin{equation}
    \theta'(0,\tau)=\frac{\partial\theta(v,\tau)}{\partial v}|_{v=0}.
\end{equation} \noindent Then the following Jacobi identity
(cf. \cite{Ch}) holds,
\begin{equation}   \theta'(0,\tau)=\pi\theta_1(0,\tau)\theta_2(0,\tau)\theta_3(0,\tau).
\end{equation}
\noindent Denote $$SL_2({\bf Z})=\left\{\left(\begin{array}{cc}
\ a & b  \\
 c  & d
\end{array}\right)\mid a,b,c,d \in {\bf Z},~ad-bc=1\right\}$$ the
modular group. Let $S=\left(\begin{array}{cc}
\ 0 & -1  \\
 1  & 0
\end{array}\right),~T=\left(\begin{array}{cc}
\ 1 &  1 \\
 0  & 1
\end{array}\right)$ be the two generators of $SL_2(\bf{Z})$. They
act on $\textbf{H}$ by $S\tau=-\frac{1}{\tau},~T\tau=\tau+1$.

\noindent
 \noindent {\bf Definition 2.1} A modular form over $\Gamma$, a
 subgroup of $SL_2({\bf Z})$, is a holomorphic function $f(\tau)$ on
 $\textbf{H}$ such that
 \begin{equation}
    f(g\tau):=f\left(\frac{a\tau+b}{c\tau+d}\right)=\chi(g)(c\tau+d)^kf(\tau),
 ~~\forall g=\left(\begin{array}{cc}
\ a & b  \\
 c & d
\end{array}\right)\in\Gamma,
 \end{equation}
\noindent where $\chi:\Gamma\rightarrow {\bf C}^{\star}$ is a
character of $\Gamma$. $k$ is called the weight of $f$.\\
Let $$\Gamma_0(2)=\left\{\left(\begin{array}{cc}
\ a & b  \\
 c  & d
\end{array}\right)\in SL_2({\bf Z})\mid c\equiv 0~({\rm
mod}~2)\right\},$$
$$\Gamma^0(2)=\left\{\left(\begin{array}{cc}
\ a & b  \\
 c  & d
\end{array}\right)\in SL_2({\bf Z})\mid b\equiv 0~({\rm
mod}~2)\right\},$$
be the two modular subgroups of $SL_2({\bf Z})$.
It is known that the generators of $\Gamma_0(2)$ are $T,~ST^2ST$,
the generators of $\Gamma^0(2)$ are $STS,~T^2STS$ (cf. \cite{Ch}).

\section{Twisted Dirac operator and Witten rigidity theorem in even dimensions}
Let $M$ be a $4k$-dimensional spin manifold and $\triangle(M)$ be the spinor bundle. Let $\widetilde{T_{\mathbf{C}}M}=T_{\mathbf{C}}M-\dim M$.
 Set
 \begin{equation}
   \Theta_1(T_{\mathbf{C}}M)=
   \bigotimes _{n=1}^{\infty}S_{q^n}(\widetilde{T_{\mathbf{C}}M})\otimes
\bigotimes _{m=1}^{\infty}\Lambda_{q^m}(\widetilde{T_{\mathbf{C}}M})
,\end{equation}
\begin{equation}
\Theta_2(T_{\mathbf{C}}M)=\bigotimes _{n=1}^{\infty}S_{q^n}(\widetilde{T_{\mathbf{C}}M})\otimes
\bigotimes _{m=1}^{\infty}\Lambda_{-q^{m-\frac{1}{2}}}(\widetilde{T_{\mathbf{C}}M}),
\end{equation}
\begin{equation}
\Theta_3(T_{\mathbf{C}}M)=\bigotimes _{n=1}^{\infty}S_{q^n}(\widetilde{T_{\mathbf{C}}M})\otimes
\bigotimes _{m=1}^{\infty}\Lambda_{q^{m-\frac{1}{2}}}(\widetilde{T_{\mathbf{C}}M}).
\end{equation}
Let $V$ be a rank $2l$ real vector bundle on $M$. Moreover, $V_{\mathbf{C}}=V\otimes\mathbf{C}.$  Set
\begin{equation}
Q_1(V_{\mathbf{C}})=\Delta(V)\otimes\bigotimes_{n=1}^{\infty}\Lambda_{q^n}(\widetilde{V_{\mathbf{C}}}),
\end{equation}
\begin{equation}
Q_2(V_{\mathbf{C}})=\bigotimes_{n=1}^{\infty}\Lambda_{-q^{n-\frac{1}{2}}}(\widetilde{V_{\mathbf{C}}}),
\end{equation}
\begin{equation}
Q_3(V_{\mathbf{C}})=\bigotimes_{n=1}^{\infty}\Lambda_{q^{n-\frac{1}{2}}}(\widetilde{V_{\mathbf{C}}}).
\end{equation}

Let
\begin{equation}
  \Phi_0(M)=\Delta(M)\otimes \Theta_1(T_{\mathbf{C}}M)+2^{2k}\Theta_2(T_{\mathbf{C}}M)+2^{2k}\Theta_3(T_{\mathbf{C}}M),
\end{equation}
\begin{equation}
  \Phi(M,V)=\Phi_0\otimes(Q_1(V))\otimes(Q_2(V))\otimes(Q_3(V)).
\end{equation}

Let $M$ be a closed smooth spin Riemannian manifold which admits a circle action. Without loss of generality, we may assume that $S^1$ acts on $M$ isometrically and preserves the spin structure of $M$.

Let $V$ be an oriented real rank-$2l$ vector bundle on a manifold $M$, and $V$ is equipped with an $S^1$ action which restricts on each $V$-fiber over $M$ to a linear action preserving that fiber. Then the associated twisted Dirac operator is $S^1$-equivariant, which implies that the corresponding orthogonal
projection $P_+$ is also $S^1$-equivariant.

Next we assume $g=e^{2\pi \mathbf{i}t}\in S^1$ be a generator of the action group and let $F_{\alpha}$ denote the fixed submanifold of the circle action on $M$. In general, $F_{\alpha}$ is not connected. Let $\mathbf{N}$ denote the normal bundle to $F_{\alpha}$ in $M$, which can be identified as the orthogonal complement of $TF_{\alpha}$ in $TM|_{F_{\alpha}}$. Then we have the following $S^1-$equivariant decomposition when restricted upon $F_{\alpha}$,
\begin{equation}
  TM|_{F_{\alpha}}=\mathbf{N}_1\oplus\cdots\oplus \mathbf{N}_{2r}\oplus TF_{\alpha},
\end{equation}
where each $\mathbf{N}_{\beta},$~$\beta=1,\cdots,2r$ is a complex vector bundle, and that $g$ acts on $\mathbf{N}_{\beta}$
by $e^{2\pi \mathbf{i}m_it}$. We assume $\{\pm2\pi\mathbf{i}x_i, \ 1\leq i\leq2r\}$ be the Chern roots of $\mathbf{N}_{\beta}\otimes\mathbf{C}$. Let $\{\pm2\pi\mathbf{i}y_j, \ 1\leq j\leq2s\}$ be the Chern roots of $TF_{\alpha}\otimes\mathbf{C}$.

Similarly, let
\begin{equation}
  V|_{F_\alpha}=V_1\oplus\cdots\oplus V_l\oplus V^{\mathbb{R}}_0
\end{equation}
be the $S^1$-equivariant decomposition of the restrictions of $V$ over $F_\alpha$, where $V_\nu$ is a complex vector bundle and $V^{\mathbb{R}}_0$ is the real subbundle of $V|_{F_\alpha}$. Assume that $g$ acts on $V_\nu$ by $e^{2\pi \mathbf{i}n_\nu t}$. And let $\{\pm2\pi\mathbf{i}z_\nu, \ 1\leq\nu\leq l\}$ be the Chern roots of $V_\nu\otimes\mathbf{C}$. Let $\{\pm2\pi\mathbf{i}z^0_\nu\}$ be the Chern roots of $V^{\mathbb{R}}_0\otimes\mathbf{C}$.

We suppose $p_1(\cdot)_{S^1}$ denote the first $S^1$-equivariant pontrjagin class, then we have the following rigidity theorems.
\begin{thm}
  For an even dimensional connected spin manifold with non-trivial $S^1$-action, if $3p_1(V)_{S^1}=0$, then the twisted Dirac operators $\mathcal{D}\otimes\Phi$ are rigid.
\end{thm}

First, we calculate the corresponding Lefschetz numbers.
\begin{prop}
 The Lefschetz numbers $\mathcal{D}\otimes\Phi$ are
\begin{equation}
  \begin{split}
  \mathcal{L}(g;\tau)&=2^{2s+2r+l}\left(\frac{1}{\pi}\right)^{2r}\sum^{2s}_{\alpha=1}
  \int_{F_\alpha}\prod^{2s}_{j=1}
  \left(y_j\frac{\theta'(0,\tau)}{\theta(y_j,\tau)}\right)\\
  &\cdot\prod^{2r}_{i=1}
  \left(\frac{\theta'(0,\tau)}{\theta(x_i+m_it,\tau)}\right)
  \sum_{\omega=1}^3\left(\prod^{2s}_{j=1}
\frac{\theta_\omega(y_j,\tau)}{\theta_\omega(0,\tau)}\prod^{2r}_{i=1}
 \frac{\theta_\omega(x_i+m_it,\tau)}{\theta_\omega(0,\tau)}\right)\\
 &\cdot\prod_{\nu=1}^l\frac{\theta_1(
 z_\nu+\eta_\nu t,\tau)\theta_2(
 z_\nu+\eta_\nu t,\tau)\theta_3(
 z_\nu+\eta_\nu t,\tau)}{\theta_1(0,\tau)\theta_2(0,\tau)\theta_3(0,\tau)},
  \end{split}
\end{equation}
\end{prop}
\begin{proof}
 By Lefschetz fixed point formula
 \begin{equation}
   \mathcal{L}(g)=\sum_i\int_{F_i}\hat{A}(TF_i)\left[Pf\left(2\sinh\left(\frac{\Omega^\perp}{4\pi}+
   \sqrt{-1}\frac{\Theta_j}{2}\right)\right)\right]^{-1}{\rm ch}_g(E).
 \end{equation}
 We can calculate it directly
\begin{equation}
  \hat{A}(TF_\alpha)\left[Pf\left(2\sinh\left(\frac{\Omega^\perp}{4\pi}+\sqrt{-1}
  \frac{\Theta_\beta}{2}\right)\right)\right]^{-1}=\prod_{j=1}^{2s}\frac{\pi y_j}{\sinh \pi y_j}\prod^{2r}_{i=1}\frac{1}{\sinh\pi(x_i+m_it)},
\end{equation}
\begin{equation}
 \begin{split}
  {\rm ch}_g(\Phi_0)=&\left(\frac{2}{\pi}\right)^{2s+2r}\prod_{j=1}^{2s}\sin(\pi y_j)\left(\frac{\theta'(0,\tau)}{\theta(y_j,\tau)}\right)\prod_{i=1}^{2r}\sin(\pi (x_i+m_it))\left(\frac{\theta'(0,\tau)}{\theta(x_i+m_it,\tau)}\right)\\
  &\cdot\sum_{\omega=1}^3
  \left(\prod^{2s}_{j=1}\frac{\theta_\omega(y_j,\tau)}{\theta_\omega(0,\tau)}\prod^{2r}_{i=1}
 \frac{\theta_\omega(x_i+m_it,\tau)}{\theta_\omega(0,\tau)}\right),
 \end{split}
\end{equation}
\begin{equation}
  {\rm ch}_g\left(\Delta(V)\otimes
  \bigotimes_{n=1}^{\infty}\Lambda_{q^n}(\widetilde{V_{\mathbf{C}}})\right)=2^{l}
  \prod_{\nu=1}^l\frac{\theta_1(z_\nu+\eta_\nu t,\tau)}{\theta_1(0,\tau)},
\end{equation}
\begin{equation}
 {\rm ch}_g\left(
\bigotimes_{n=1}^{\infty}\Lambda_{-q^{n-\frac{1}{2}}}(\widetilde{V_{\mathbf{C}}})\right) =\prod_{\nu=1}^l\frac{\theta_2(z_\nu+\eta_\nu t,\tau)}{\theta_2(0,\tau)},
\end{equation}
\begin{equation}
  {\rm ch}_g\left(
  \bigotimes_{n=1}^{\infty}\Lambda_{q^{n-\frac{1}{2}}}(\widetilde{V_{\mathbf{C}}})\right)=\prod_{\nu=1}^l
  \frac{\theta_3(z_\nu+\eta_\nu t,\tau)}{\theta_3(0,\tau)}.
\end{equation}
So,it's true that we get Proposition 3.2 by (3.12)-(3.17).
\end{proof}

In the follows, let us view the above expression as defining a function $\mathcal{L}'(t,\tau)$, such that $\mathcal{L}'(t,\tau)=\mathcal{L}(g;\tau)$.
The expression for $\mathcal{L}'(t,\tau)$ involves only complex-differentiable functions, so we can extend its domain to every complex number $t$ and choice of $\tau$ in the open upper half-plane, i.e. $\mathcal{L}'(t,\tau)$ where the expression exists in $\mathbf{C}\times\mathbf{H}$. The Witten rigidity theorems are equivalent to that these $\mathcal{L}'(t,\tau)$ are independent of $t$.  Then, we have the following lemma.
\begin{lem}
Let $(t,\tau)\in\mathbf{C}\times\mathbf{H}$ be in the domain of $\mathcal{L}'(t,\tau)$.\\
(1)Then $\mathcal{L}'(t,\tau)=\mathcal{L}'(t+a,\tau)$ for any $a\in 2\mathbb{Z}$.\\
(2)If $3p_1(V)_{S^1}=0$, then $\mathcal{L}'(t,\tau)=\mathcal{L}'(t+a\tau,\tau)$ for any $a\in 2\mathbb{Z}$.
\end{lem}
\begin{proof}
  We have the following transformation formulas of theta-functions[7]:
  \begin{equation}
    \theta(t+1,\tau)=-\theta(t,\tau), \ \ \theta(t+\tau,\tau)=-q^{-1/2}e^{-2\pi\mathbf{i}t}\theta(t,\tau),\nonumber
  \end{equation}
  \begin{equation}
    \theta_1(t+1,\tau)=-\theta_1(t,\tau), \ \ \theta_1(t+\tau,\tau)=q^{-1/2}e^{-2\pi\mathbf{i}t}\theta_1(t,\tau),\nonumber
  \end{equation}
  \begin{equation}
    \theta_2(t+1,\tau)=\theta_2(t,\tau), \ \ \theta_2(t+\tau,\tau)=-q^{-1/2}e^{-2\pi\mathbf{i}t}\theta_2(t,\tau),\nonumber
  \end{equation}
  \begin{equation}
    \theta_3(t+1,\tau)=\theta_3(t,\tau), \ \ \theta_3(t+\tau,\tau)=-q^{-1/2}e^{-2\pi\mathbf{i}t}\theta_3(t,\tau),\nonumber
  \end{equation}
where $q$ here is equal to $e^{\pi\mathbf{i}\tau}$.

By reusing the formula in the first column, for $a\in 2\mathbb{Z}$, we can easily verify that $\theta,\theta_1,\theta_2,\theta_3$ terms is unchanged in $\mathcal{L}'$. So (1) is true.

For (2), under the replacement $t\to t+a\tau$ for $a$ an even integer,  we apply the second column formula to get the following formula:
\begin{equation}
  \theta(t+a\tau)=(-1)^ae^{-2\pi\mathbf{i}(at+a^2\tau/2)}\theta(t,\tau),\nonumber
\end{equation}
\begin{equation}
  \theta_1(t+a\tau)=e^{-2\pi\mathbf{i}(at+a^2\tau/2)}\theta_1(t,\tau),\nonumber
\end{equation}
\begin{equation}
  \theta_2(t+a\tau)=(-1)^ae^{-2\pi\mathbf{i}(at+a^2\tau/2)}\theta_2(t,\tau),\nonumber
\end{equation}
\begin{equation}
  \theta_3(t+a\tau)=e^{-2\pi\mathbf{i}(at+a^2\tau/2)}\theta_3(t,\tau).\nonumber
\end{equation}
Then, under the replacement $t\to t+a\tau$ for $a$ an even integer, the following transformation property holds:
\begin{equation}
  \theta_\mu(z_\nu+n_\nu(t+a\tau),\tau)=e^{-2\pi\mathbf{i}a(n_\nu z_\nu+n^2_\nu(t+a^2\tau/2))}\theta_\mu(z_\nu+n_\nu t,\tau),
\end{equation}
where $\theta_\mu\in\{\theta,\theta_1,\theta_2,\theta_3\}$.

By the condition $3p_1(V)_{S^1}=0$ implies that
$$3p_1(V)_{S^1}=p_1(\otimes_3V)_{S^1}=0,$$
then, we have
$$\sum_{\nu=1}^ln_\nu z_\nu=0, \ \ \sum_{\nu=1}^ln^2_\nu=0.$$
Thus, we can see that for each connected component $F_\alpha$ of a fixed locus of $M$, $\mathcal{L}'$ is completely unchanged. (2) certification.
\end{proof}

This lemma tells us that for fixed $\tau$ these $\mathcal{L}'(t,\tau)$ are meromorphic functions on the
torus $\mathbf{C}/2\mathbb{Z}\times2\mathbb{Z}\tau$. Therefore, to get the rigidity we only need to prove that they are holomorphic in $t$. We will actually prove that they are holomorphic in $(t,\tau)$ on $\mathbf{C}\times\mathbf{H}$.

Given
$$g=\left(\begin{array}{cc}
\ a & b  \\
 c  & d
\end{array}\right)\in SL_2(\mathbb{Z})$$
define its modular transformation on $\mathbf{C}\times\mathbf{H}$ by
$$g(t,\tau)=\left(\frac{\tau}{c\tau+d},\frac{a\tau+b}{c\tau+d}\right).$$
This defines a group action. Obviously two generators of $SL_2(\mathbb{Z})$,
$$S=\left(\begin{array}{cc}
\ 0 & -1  \\
 1  & 0
\end{array}\right),~T=\left(\begin{array}{cc}
\ 1 &  1 \\
 0  & 1
\end{array}\right)$$
act by
$$S(t,\tau)=\left(\frac{t}{\tau},-\frac{1}{\tau}\right),~T(t,\tau)=(t,\tau+1).$$
Then,we have the following transformation formulas
\begin{lem}
  (1)If $3p_1(V)_{S^1}=0$, then the action of $S$,
  \begin{equation}
    \mathcal{L}'\left(\frac{t}{\tau},-\frac{1}{\tau}\right)=\tau^{2k}\mathcal{L}'(t,\tau).
  \end{equation}
  (2)The action of $T$, that
  \begin{equation}
    \mathcal{L}'(t,\tau+1)=\mathcal{L}'(t,\tau).
  \end{equation}
\end{lem}
\begin{proof}
  We have the following transformation laws of Jacobi theta-functions under the actions of $S$ and $T$ (cf. \cite{Ch}):
$$\theta\left(\frac{t}{\tau},-\frac{1}{\tau}\right)=\frac{1}{\mathbf{i}}\sqrt{\frac{\tau}{\mathbf{i}}}
e^{\pi\mathbf{i}t^2/\tau}\theta(t,\tau),~~\theta(t,\tau+1)=e^{\pi\mathbf{i}/4}\theta(t,\tau);$$
$$\theta_1\left(\frac{t}{\tau},-\frac{1}{\tau}\right)=\frac{1}{\mathbf{i}}\sqrt{\frac{\tau}{\mathbf{i}}}
e^{\pi\mathbf{i}t^2/\tau}\theta_2(t,\tau),~~\theta_1(t,\tau+1)=e^{\pi\mathbf{i}/4}\theta_1(t,\tau);$$
$$\theta_2\left(\frac{t}{\tau},-\frac{1}{\tau}\right)=\frac{1}{\mathbf{i}}\sqrt{\frac{\tau}{\mathbf{i}}}
e^{\pi\mathbf{i}t^2/\tau}\theta_1(t,\tau),~~\theta_2(t,\tau+1)=\theta_3(t,\tau);$$
$$\theta_3\left(\frac{t}{\tau},-\frac{1}{\tau}\right)=\frac{1}{\mathbf{i}}\sqrt{\frac{\tau}{\mathbf{i}}}
e^{\pi\mathbf{i}t^2/\tau}\theta_3(t,\tau),~~\theta_3(t,\tau+1)=e^{\pi\mathbf{i}/4}\theta_2(t,\tau);$$
$$\theta'\left(0,-\frac{1}{\tau}\right)=\left(\frac{\tau}{\mathbf{i}}\right)^{3/2}\theta'(0,\tau),
~~\theta'(0,\tau+1)=e^{\pi\mathbf{i}/4}\theta'(0,\tau).$$
Then, we obtain the following transformation formulas,
\begin{equation}
  y_j\frac{\theta'(0,-1/\tau)}{\theta(y_j,-1/\tau)}=e^{-\pi\mathbf{i}\tau y^2_j}\tau y_j\frac{\theta'(0,\tau)}{\theta(\tau y_j,\tau)},
\end{equation}
\begin{equation}
  \frac{\theta'(0,-1/\tau)}{\theta(x_i+m_it/\tau,-1/\tau)}=e^{-\pi\mathbf{i}\tau(x_i+m_it/\tau)^2}\tau \frac{\theta'(0,\tau)}{\theta(\tau x_i+m_it,\tau)},
\end{equation}
\begin{equation}
  \frac{\theta_1(y_j,-1/\tau)}{\theta_1(0,-1/\tau)}=e^{\pi\mathbf{i}\tau y^2_j}\frac{\theta_2(\tau y_j,\tau)}{\theta_2(0,\tau)},
\end{equation}
\begin{equation}
  \frac{\theta_1(x_i+m_it/\tau,-1/\tau)}{\theta_1(0,-1/\tau)}=e^{\pi\mathbf{i}\tau (x_i+m_it/\tau)^2}\frac{\theta_2(\tau x_i+m_it,\tau)}{\theta_2(0,\tau)},
\end{equation}
\begin{equation}
  \frac{\theta_2(y_j,-1/\tau)}{\theta_2(0,-1/\tau)}=e^{\pi\mathbf{i}\tau y^2_j}\frac{\theta_1(\tau y_j,\tau)}{\theta_1(0,\tau)},
\end{equation}
\begin{equation}
  \frac{\theta_2(x_i+m_it/\tau,-1/\tau)}{\theta_2(0,-1/\tau)}=e^{\pi\mathbf{i}\tau (x_i+m_it/\tau)^2}\frac{\theta_1(\tau x_i+m_it,\tau)}{\theta_1(0,\tau)},
\end{equation}
\begin{equation}
  \frac{\theta_3(y_j,-1/\tau)}{\theta_3(0,-1/\tau)}=e^{\pi\mathbf{i}\tau y^2_j}\frac{\theta_3(\tau y_j,\tau)}{\theta_3(0,\tau)},
\end{equation}
\begin{equation}
  \frac{\theta_3(x_i+m_it/\tau,-1/\tau)}{\theta_3(0,-1/\tau)}=e^{\pi\mathbf{i}\tau (x_i+m_it/\tau)^2}\frac{\theta_3(\tau x_i+m_it,\tau)}{\theta_3(0,\tau)},
\end{equation}
\begin{equation}
  \frac{\theta_1(z_\nu+n_\nu t/\tau,-1/\tau)}{\theta_1(0,-1/\tau)}=e^{\pi\mathbf{i}\tau (z_\nu+n_\nu t/\tau)^{2}}\frac{\theta_2(\tau z_\nu+n_\nu t,\tau)}{\theta_2(0,\tau)}.
\end{equation}
Combining the condition $3p_1(V)_{S^1}=0$, and (3.23)-(3.31), we obtain (3.19).

For (2), we use the transformation laws of Jacobi theta-functions under the action of $T$ , which we can easily verify to get (3.20).
\end{proof}

\begin{lem}
  For any function
  $$\mathcal{L}'(t,\tau)$$
  its modular transformation is holomorphic in $(t,\tau)\in\mathbf{R}\times\mathbf{H}$.
\end{lem}
\begin{proof}
The proof is almost the same as the proof of [12, Lemma 2.3] except that we use Proposition 3.2 instead of the corresponding Lefschetz fixed point formulas.
\end{proof}

$Proof~of~Theorem~3.1$. We prove that these $\mathcal{F}_\lambda$ are holomorphic on $\mathbf{C}\times\mathbf{H}$ which implies the rigidity of Theorem 3.1. By (3.11), we see that
the possible poles of $\mathcal{L}'(t,\tau)$ can be written in the form $t=\frac{k}{l}(c\tau+d)$ for
integers $k,l,c,d$  with $(c, d)=1$.

We can always find integers $a,b$ such that $ad-bc=1$. Then the matrix $g_0=\left(\begin{array}{cc}
\ d & -b  \\
 -c  & a
\end{array}\right)\in SL_2(\mathbb{Z})$ induces an action
$$\mathcal{L}'(g_0(t,\tau))=\mathcal{F}\left(\frac{t}{-c\tau+a},\frac{d\tau-b}{-c\tau+a}\right).$$
Now, if $t=\frac{k}{l}(c\tau+d)$ is a polar divisor of $\mathcal{F}(t,\tau)$, then one polar divisor of
$\mathcal{L}'(g_0(t,\tau))$ is given by
$$\frac{t}{-c\tau+a}=\frac{k}{l}\left(c\frac{d\tau-b}{-c\tau+a}+d\right)$$
which gives $t=\frac{k}{l}$.

By Lemma 3.4,  we know that up to some constant, $\mathcal{L}'(g_0(t,\tau))$ is still one of $\{\mathcal{L}'(t,\tau)\}$. This contradicts Lemma 3.5, therefore, this completes the
proof of Theorem 3.1.

\section{Twisted Toplitz operator and Witten rigidity theorem in odd dimensions}
We recall the odd Chern character  of a smooth map g from M to the general linear group $GL(N,\mathbf{C})$
with $N$ a  positive integer (see \cite{HY}). Let $d$ denote a trivial connection on $\mathbf{C}^{N}|_{M}$. We will denote by $c_g(M,[g])$ the cohomology class associated to the closed $n$-form
\begin{equation}
  c_n(\mathbf{C}^{N}|_M,g,d)=\left(\frac{1}{2\pi\sqrt{-1}}\right)^{\frac{(n+1)}{2}}\mathrm{Tr}[(g^{-1}dg)^n].
\end{equation}
The odd Chern character form ${\rm ch}(\mathbf{C}^{N}|_M,g,d)$ associated to $g$ and $d$ by definition is
\begin{equation}
  {\rm ch}(\mathbf{C}^{N}|_M,g,d)=\sum^{\infty}_{n=1}\frac{n!}{(2n+1)!}c_{2n+1}((\mathbf{C}^{N}|_M,g,d)).
\end{equation}
Let the connection $\nabla_{u}$ on the trivial bundle $\mathbf{C}^{N}|_M$ defined by
\begin{equation}
  \nabla_u=(1-u)d+ug^{-1}\cdot d \cdot g,\ \ u\in[0,1].
\end{equation}
Then we have
\begin{equation}
  d{\rm ch}(\mathbf{C}^{N}|_M,g,d)={\rm ch}(\mathbf{C}^{N}|_M,d)-{\rm ch}(\mathbf{C}^{N}|_M,g^{-1}\cdot d\cdot g).
\end{equation}

Now let $g:M\to SO(N)$ and we assume that $N$ is even and large enough. Let $E$ denote the trivial real vector bundle of rank $N$ over $M$. We equip $E$ with the canonical trivial metric and trivial connection $d$. Set
$$\nabla_u=d+ug^{-1}dg,\ \ u\in[0,1].$$
Let $R_u$ be the curvature of $\nabla_u$, then
\begin{equation}
  R_u=(u^2-u)(g^{-1}dg)^2.
\end{equation}
We also consider the complexification of $E$ and $g$ extends to a unitary automorphism of $E_{\mathbf{C}}$. The connection $\nabla_u$ extends to a Hermitian connection on $E_{\mathbf{C}}$ with curvature still given by (3.6). Let $\Delta(E)$ be the spinor bundle of $E$, which is a trivial Hermitian
bundle of rank $2^{\frac{N}{2}}$. We assume that $g$ has a lift to the Spin group ${\rm Spin}(N):g^{\Delta}:M\to {\rm Spin}(N)$. So $g^{\Delta}$ can be viewed as an automorphism of $\Delta(E)$ preserving the Hermitian metric. We lift $d$ on $E$ to be a trivial Hermitian connection $d^{\Delta}$ on $\Delta(E)$, then
\begin{equation}
  \nabla_u^{\Delta}=(1-u)d^{\Delta}+u(g^{\Delta})^{-1}\cdot d^{\Delta} \cdot g^{\Delta},\ \ u\in[0,1]
\end{equation}
lifts $\nabla_u$ on $E$ to $\Delta(E)$. Let $Q_j(E),j=1,2,3$ be the virtual bundles defined as follows:
\begin{equation}
Q_1(E)=\triangle(E)\otimes
   \bigotimes _{n=1}^{\infty}\wedge_{q^n}(\widetilde{E_C});
\end{equation}
\begin{equation}
Q_2(E)=\bigotimes _{n=1}^{\infty}\wedge_{-q^{n-\frac{1}{2}}}(\widetilde{E_C});
\end{equation}
\begin{equation}
Q_3(E)=\bigotimes _{n=1}^{\infty}\wedge_{q^{n-\frac{1}{2}}}(\widetilde{E_C}).
\end{equation}
Let $g$ on $E$ have a lift $g^{Q(E)}$ on $Q(E)$ and $\nabla_u$ have a lift $\nabla^{Q(E)}_u$ on $Q(E)$. Following \cite{HY}, we defined ${\rm ch}(Q_j(E),g^{Q_j(E)},d,\tau)$ for $j=1,2,3$ as following
\begin{equation}
{\rm ch}(Q_j(E),\nabla^{Q_j(E)}_0,\tau)-{\rm ch}(Q_j(E),\nabla^{Q_j(E)}_1,\tau)=d{\rm ch}(Q_j(E),g^{Q_j(E)},d,\tau),
\end{equation}
where
\begin{equation}
{\rm ch}(Q_1(E),g^{Q_1(E)},d,\tau)=-\frac{2^{N/2}}{8\pi^2}\int^1_0{\rm Tr}\left[g^{-1}dg\frac{\theta'_1(R_u/(4\pi^2),\tau)}{\theta_1(R_u/(4\pi^2),\tau)}\right]du,
\end{equation}
and for $j=2,3$
\begin{equation}
{\rm ch}(Q_j(E),g^{Q_j(E)},d,\tau)=-\frac{1}{8\pi^2}\int^1_0{\rm Tr}\left[g^{-1}dg\frac{\theta'_j(R_u/(4\pi^2),\tau)}{\theta_j(R_u/(4\pi^2),\tau)}\right]du.
\end{equation}

Let $M$ be a $(4k-1)$-dimensional spin manifold and $\triangle(M)$ be the spinor bundle. Let $\widetilde{T_{\mathbf{C}}M}=T_{\mathbf{C}}M-\dim M$.
Set
\begin{equation}
  \Psi_1(M,V,E)=\Phi_0\otimes\left(\Delta(V)\otimes
  \bigotimes_{n=1}^{\infty}\Lambda_{q^n}(\widetilde{V_{\mathbf{C}}})\right)\otimes(Q_1(E),g^{Q_1(E)}),
\end{equation}
\begin{equation}
  \Psi_2(M,V,E)=\Phi_0\otimes\left(
  \bigotimes_{n=1}^{\infty}\Lambda_{-q^{n-\frac{1}{2}}}(\widetilde{V_{\mathbf{C}}})\right)\otimes
  (Q_2(E),g^{Q_2(E)}),
\end{equation}
\begin{equation}
  \Psi_3(M,V,E)=\Phi_0\otimes\left(
  \bigotimes_{n=1}^{\infty}\Lambda_{q^{n-\frac{1}{2}}}(\widetilde{V_{\mathbf{C}}})\right)\otimes
  (Q_3(E),g^{Q_3(E)}).
\end{equation}

Let $M$ be an odd dimensional closed smooth spin Riemannian manifold which admits a circle action.
Let $V$ be an $S^1$-equivariant complex vector bundle over $M$ carrying an $S^1$-invariant Hermitian connection. In addition, we assume $g:M\to GL(N,\mathbf{C})$ is $S^1$-invariant, i.e.,
\begin{equation}
  g(hx)=g(x), \ for \ any \ h\in S^1 \ and \ x\in M.
\end{equation}
Thus the twisted Toeplitz operator $\mathcal{T}\otimes\Psi_i$,~$i=1,2,3$ is $S^1$-equivariant.

We consider the fixed point case. Similarly, we have the following $S^1$-equivariant decomposition when restricted upon $F_{\alpha}$,
\begin{equation}
  TM|_{F_{\alpha}}=\mathbf{N}_1\oplus\cdots\oplus \mathbf{N}_{2r}\oplus TF_{\alpha},
\end{equation}
where each $\mathbf{N}_{\beta},$~$\beta=1,\cdots,2r$ is a complex vector bundle such that $h\in S^1$ acts on it
by $e^{2\pi \mathbf{i}m_it}$. We assume $\{\pm2\pi\mathbf{i}x_i, \ 1\leq i\leq2r\}$ be the Chern roots of $\mathbf{N}_{\beta}\otimes\mathbf{C}$. Let $\{\pm2\pi\mathbf{i}y_j, \ 1\leq j\leq2s\}$ be the Chern roots of $TF_{\alpha}\otimes\mathbf{C}$.
And let
\begin{equation}
  V|_{F_\alpha}=V_1\oplus\cdots\oplus V_l\oplus V^{\mathbb{R}}_0
\end{equation}
be the $S^1$-equivariant decomposition of the restrictions of $V$ over $F_\alpha$, where $V_\nu$ is a complex vector bundle and $V^{\mathbb{R}}_0$ is the real subbundle of $V|_{F_\alpha}$. Assume that $h\in S^1$ acts on $V_\nu$ by $e^{2\pi \mathbf{i}n_\nu t}$. And let $\{\pm2\pi\mathbf{i}z_\nu, \ 1\leq\nu\leq l\}$ be the Chern roots of $V_\nu\otimes\mathbf{C}$. Let $\{\pm2\pi\mathbf{i}z^0_\nu\}$ be the Chern roots of $V^{\mathbb{R}}_0\otimes\mathbf{C}$.

We suppose $p_1(\cdot)_{S^1}$ denote the first $S^1$-equivariant pontrjagin class, then we have the following rigidity theorems.
\begin{thm}
  For an odd dimensional connected spin manifold with non-trivial $S^1$-action, if $p_1(V)_{S^1}=0$ and $c_3(E,g,d)=0$, then the Toeplitz operators $\mathcal{T}\otimes\Psi_i,~~i=1,2,3$ are rigid.
\end{thm}

First, we calculate the corresponding Lefschetz numbers.
\begin{prop}
 The Lefschetz numbers $\mathcal{T}\otimes\Psi_i$,~$i=1,2,3$ are
\begin{equation}
  \begin{split}
  \mathcal{L}_1(g;\tau)&=2^{2s+2r+al}\left(\frac{1}{\pi}\right)^{2r}\sum^{2s}_{\alpha=1}
  \int_{F_\alpha}\prod^{2s}_{j=1}
  \left(y_j\frac{\theta'(0,\tau)}{\theta(y_j,\tau)}\right)\\
  &\cdot\prod^{2r}_{i=1}
  \left(\frac{\theta'(0,\tau)}{\theta(x_i+m_it,\tau)}\right)
  \sum_{\omega=1}^3\left(\prod^{2s}_{j=1}
\frac{\theta_\omega(y_j,\tau)}{\theta_\omega(0,\tau)}\prod^{2r}_{i=1}
 \frac{\theta_\omega(x_i+m_it,\tau)}{\theta_\omega(0,\tau)}\right)\\
 &\cdot\prod_{\nu=1}^l\frac{\theta_1^a(
 z_\nu+\eta_\nu t,\tau)}{\theta_1^a(0,\tau)}{\rm ch}(Q_1(E),g^{Q_1(E)},d,\tau),
  \end{split}
\end{equation}
\begin{equation}
  \begin{split}
  \mathcal{L}_2(g;\tau)&=2^{2s+2r}\left(\frac{1}{\pi}\right)^{2r}\sum^{2s}_{\alpha=1}
  \int_{F_\alpha}\prod^{2s}_{j=1}
  \left(y_j\frac{\theta'(0,\tau)}{\theta(y_j,\tau)}\right)\\
  &\cdot\prod^{2r}_{i=1}
  \left(\frac{\theta'(0,\tau)}{\theta(x_i+m_it,\tau)}\right)
  \sum_{\omega=1}^3\left(\prod^{2s}_{j=1}
\frac{\theta_\omega(y_j,\tau)}{\theta_\omega(0,\tau)}\prod^{2r}_{i=1}
 \frac{\theta_\omega(x_i+m_it,\tau)}{\theta_\omega(0,\tau)}\right)\\
 &\cdot\prod_{\nu=1}^l\frac{\theta_2^b(
 z_\nu+\eta_\nu t,\tau)}{\theta_2^b(0,\tau)}{\rm ch}(Q_2(E),g^{Q_2(E)},d,\tau),
  \end{split}
\end{equation}
and
\begin{equation}
  \begin{split}
  \mathcal{L}_3(g;\tau)&=2^{2s+2r}\left(\frac{1}{\pi}\right)^{2r}\sum^{2s}_{\alpha=1}
  \int_{F_\alpha}\prod^{2s}_{j=1}
  \left(y_j\frac{\theta'(0,\tau)}{\theta(y_j,\tau)}\right)\\
  &\cdot\prod^{2r}_{i=1}
  \left(\frac{\theta'(0,\tau)}{\theta(x_i+m_it,\tau)}\right)
  \sum_{\omega=1}^3\left(\prod^{2s}_{j=1}
\frac{\theta_\omega(y_j,\tau)}{\theta_\omega(0,\tau)}\prod^{2r}_{i=1}
 \frac{\theta_\omega(x_i+m_it,\tau)}{\theta_\omega(0,\tau)}\right)\\
 &\cdot\prod_{\nu=1}^l\frac{\theta_3^c(
 z_\nu+\eta_\nu t,\tau)}{\theta_3^c(0,\tau)}{\rm ch}(Q_3(E),g^{Q_3(E)},d,\tau).
  \end{split}
\end{equation}
\end{prop}
\begin{proof}
Similar to Proposition 3.2, we can calculate equations (4.19), (4.20) and (4.21) by using Lefschetz fixed point formula, (4.11) and (4.12).
\end{proof}

Next, let us view the above expression as defining a function $\mathcal{F}_\lambda(t,\tau)$, for each $\lambda\in\{1,2,3\}$, i.e. $\mathcal{F}_\lambda(t,\tau)=\mathcal{L}_\lambda(g;\tau)$.
The expression for $\mathcal{F}_\lambda(t,\tau)$ involves only complex-differentiable functions, so we can extend its domain to every complex number $t$ and choice of $\tau$ in the open upper half-plane, i.e. $\mathcal{F}_\lambda(t,\tau)$ where the expression exists in $\mathbf{C}\times\mathbf{H}$. The Witten rigidity theorems are equivalent to that these $\mathcal{F}_\lambda(t,\tau)$ are independent of $t$.  Then, we have the following lemma.
\begin{lem}
Let $(t,\tau)\in\mathbf{C}\times\mathbf{H}$ be in the domain of $\mathcal{F}_\lambda(t,\tau)$,~$\lambda\in\{1,2,3\}$.\\
(1)Then $\mathcal{F}_\lambda(t,\tau)=\mathcal{F}_\lambda(t+a,\tau)$ for any $a\in 2\mathbb{Z}$.\\
(2)If $p_1(V)_{S^1}=0$, then $\mathcal{F}_\lambda(t,\tau)=\mathcal{F}_\lambda(t+a\tau,\tau)$ for any $a\in 2\mathbb{Z}$.
\end{lem}
\begin{proof}
Similar to Lemma 3.3.
\end{proof}

Next, we will actually prove that these $\mathcal{F}_\lambda$ are holomorphic in $(t,\tau)$ on $\mathbf{C}\times\mathbf{H}$. Similar to Lemma 3.4, we have the following lemma
\begin{lem}
  (1)If $p_1(V)_{S^1}=0,$ and $c_3(E,g,d)=0$, then the action of $S$,
  \begin{equation}
    \mathcal{F}_1\left(\frac{t}{\tau},-\frac{1}{\tau}\right)=2^{al+N/2}\tau^{2k}\mathcal{F}_2(t,\tau),
  \end{equation}
  \begin{equation}
    \mathcal{F}_2\left(\frac{t}{\tau},-\frac{1}{\tau}\right)=2^{-(al+N/2)}\tau^{2k}\mathcal{F}_1(t,\tau),
  \end{equation}
  \begin{equation}
    \mathcal{F}_3\left(\frac{t}{\tau},-\frac{1}{\tau}\right)=\tau^{2k}\mathcal{F}_3(t,\tau).
  \end{equation}
  (2)The action of $T$, then
  \begin{equation}
    \mathcal{F}_1(t,\tau+1)=\mathcal{F}_1(t,\tau),
  \end{equation}
   \begin{equation}
    \mathcal{F}_2(t,\tau+1)=\mathcal{F}_3(t,\tau),
  \end{equation}
   \begin{equation}
    \mathcal{F}_3(t,\tau+1)=\mathcal{F}_2(t,\tau).
  \end{equation}
\end{lem}
\begin{proof}
By Proposition 2.2 in \cite{HY}, we have if $c_3(E,g,d)=0$, then for any integer $i\geq 1$,
\begin{equation}
  \left\{{\rm ch}\left(Q_1(E),g^{Q_1(E)},d,-\frac{1}{\tau}\right)\right\}^{4i-1}=2^{N/2}\left\{\tau^{2i}{\rm ch}\left(Q_2(E),g^{Q_2(E)},d,\tau\right)\right\}^{4i-1},
\end{equation}
\begin{equation}
  \left\{{\rm ch}\left(Q_2(E),g^{Q_2(E)},d,-\frac{1}{\tau}\right)\right\}^{4i-1}=2^{-N/2}\left\{\tau^{2i}{\rm ch}\left(Q_1(E),g^{Q_1(E)},d,\tau\right)\right\}^{4i-1},
\end{equation}
\begin{equation}
  \left\{{\rm ch}\left(Q_3(E),g^{Q_3(E)},d,-\frac{1}{\tau}\right)\right\}^{4i-1}=\left\{\tau^{2i}{\rm ch}\left(Q_3(E),g^{Q_3(E)},d,\tau\right)\right\}^{4i-1}.
\end{equation}
Combining the condition $p_1(V)_{S^1}=0$ and (3.21)-(3.29), we obtain (4.22). The formulas (4.23) and (4.24) can be verified in a similar way.

For (2), we use the transformation laws of Jacobi theta-functions under the action of $T$ and Proposition 2.2 in \cite{HY}, which we can easily verify to get (4.25), (4.26) and (4.27).
\end{proof}

\begin{lem}
  For any function
  $$\mathcal{F}_\lambda, \ \ \lambda=\{1,2,3\}$$
  its modular transformation is holomorphic in $(t,\tau)\in\mathbf{R}\times\mathbf{H}$.
\end{lem}
\begin{proof}
The proof is almost the same as the proof of [12, Lemma 2.3] except that we use Proposition 4.2 instead of the corresponding Lefschetz fixed point formulas.
\end{proof}

$Proof~of~Theorem~4.1$. We prove that these $\mathcal{F}_\lambda$ are holomorphic on $\mathbf{C}\times\mathbf{H}$ which implies the rigidity of Theorem 4.1. We denote by $\mathcal{F}$ one of the functions: $\{\mathcal{F}_1,\mathcal{F}_2,\mathcal{F}_3\}$. By (4.19), (4.20) and (4.21), we see that
the possible poles of $\mathcal{F}(t,\tau)$ can be written in the form $t=\frac{k}{l}(c\tau+d)$ for
integers $k,l,c,d$  with $(c, d)=1$.

We can always find integers $a,b$ such that $ad-bc=1$. Then the matrix $g_1=\left(\begin{array}{cc}
\ d & -b  \\
 -c  & a
\end{array}\right)\in SL_2(\mathbb{Z})$ induces an action
$$\mathcal{F}(g_1(t,\tau))=\mathcal{F}\left(\frac{t}{-c\tau+a},\frac{d\tau-b}{-c\tau+a}\right).$$
Now, if $t=\frac{k}{l}(c\tau+d)$ is a polar divisor of $\mathcal{F}(t,\tau)$, then one polar divisor of
$\mathcal{F}(g_1(t,\tau))$ is given by
$$\frac{t}{-c\tau+a}=\frac{k}{l}\left(c\frac{d\tau-b}{-c\tau+a}+d\right)$$
which gives $t=\frac{k}{l}$.

By Lemma 4.4,  we know that up to some constant, $\mathcal{F}(g_1(t,\tau))$ is still one of $\{\mathcal{F}_1,\mathcal{F}_2,\mathcal{F}_3\}$. This contradicts Lemma 4.5, therefore, this completes the
proof of Theorem 4.1.

Similarly, use Liu's method \cite{Li2}. In the odd-dimensional case, we can get the following theorem
\begin{thm}
For an odd dimensional connected spin manifold with non-trivial $S^1$-action.\\
(1)If $3p_1(V)_{S^1}=0$ and $c_3(E,g,d)=0$, then the Toeplitz operators $\mathcal{T}\otimes\Phi\otimes(Q_1(E),g^{Q_1(E)})$,
$\mathcal{T}\otimes\Phi\otimes(Q_2(E),g^{Q_2(E)})$, $\mathcal{T}\otimes\Phi\otimes(Q_3(E),g^{Q_3(E)})$
and $\mathcal{T}\otimes\Phi\otimes(Q(E),g^{Q(E)})$ are rigid.\\
(2)If $p_1(V)_{S^1}=0$ and $c_3(E,g,d)=0$, then the Toeplitz operators $\mathcal{T}\otimes\Psi_i\otimes(Q(E),g^{Q(E)})$ are rigid.
Where $Q(E)=Q_1(E)\otimes Q_2(E)\otimes Q_3(E)$.
\end{thm}

\section{Acknowledgements}

 The second author was supported in part by NSFC No.11771070. The authors also thank the referee for his (or her) careful reading and helpful comments.

\vskip 1 true cm

\section{Data availability}

No data was gathered for this article.

\section{Conflict of interest}

The authors have no relevant financial or non-financial interests to disclose.

\vskip 1 true cm

\bigskip
\bigskip
\indent{J. Guan}\\
 \indent{School of Mathematics and Statistics,
Northeast Normal University, Changchun Jilin, 130024, China }\\
\indent E-mail: {\it guanjy@nenu.edu.cn }\\
\indent{K. Liu}\\
 \indent{Mathematical Science Research Center,Chongqing University of Technology, Chongqing, 400054, China }\\
\indent E-mail: {\it kefeng@cqut.edu.cn }\\
\indent{Y. Wang}\\
 \indent{School of Mathematics and Statistics,
Northeast Normal University, Changchun Jilin, 130024, China }\\
\indent E-mail: {\it wangy581@nenu.edu.cn }\\

\end{document}